\documentclass[12pt]{article}
\usepackage{amsmath,amssymb,amsthm,amsfonts,mathrsfs}
\usepackage{appendix}
\usepackage{enumitem,graphics,xcolor,yfonts,colonequals}
\usepackage{stmaryrd}
\usepackage[alphabetic]{amsrefs}
\usepackage[all,cmtip]{xy}
\usepackage{tikz-cd}
\usepackage[colorlinks,anchorcolor=blue,citecolor=blue,linkcolor=blue,urlcolor=blue,bookmarksopenlevel=1]{hyperref}
\usepackage[margin=1in]{geometry}

\usepackage[]{cleveref}

\usepackage{enumitem}

\usepackage{csquotes}

\usepackage{fancyhdr}

\fancypagestyle{titlepage}
{
	\fancyhf{}

	\fancyfoot[l]{
	\href{https://mathscinet.ams.org/mathscinet/msc/msc2020.html}
	}
}

\AtBeginDocument{%
	\def\MR#1{}
}

\newtheorem*{theorem*}{Theorem}
\newtheorem*{proposition*}{Proposition}
\newtheorem*{corollary*}{Corollary}
\newtheorem{theorem}{Theorem}[section]
\newtheorem{proposition}[theorem]{Proposition}
\newtheorem{corollary}[theorem]{Corollary}
\newtheorem{lemma}[theorem]{Lemma}

\numberwithin{equation}{section}

\theoremstyle{definition}

\newtheorem{example}[theorem]{Example}

\newtheorem{remark}[theorem]{Remark}

\DeclareMathOperator{\adj}{adj}
\DeclareMathOperator{\pr}{pr}

\renewcommand{\P}{\mathbb{P}}
\newcommand{\etasym}{\eta_{\rm sym}}
\newcommand{\C}{\mathbb{C}}
\newcommand{\ol}{\overline}

\newcommand{\sO}{\mathcal{O}}

\begin{document}
\title{Plane quartics and heptagons}
\author{Daniele Agostini, Daniel Plaumann, Rainer Sinn, Jannik Lennart Wesner}
\date{\today}

\maketitle
\thispagestyle{titlepage}

%--------------------Abstract
\begin{abstract}
Every polygon with $n$ vertices in the complex projective plane is naturally associated with its adjoint curve of degree $n-3$. Hence the adjoint of a heptagon is a plane quartic. We prove that a general plane quartic is the adjoint of exactly 864 distinct complex heptagons. This number had been numerically computed by Kohn et al. We use intersection theory and the Scorza correspondence for quartics to show that $864$ is an upper bound, complemented by a lower bound obtained through explicit analysis of the famous Klein quartic.
\end{abstract}

\section{Introduction}
Given a convex $n$-gon in the real plane bounded by $n$ lines in general position, there is a unique curve of degree $n-3$ passing through the $\binom{n}{2}-n$ intersection points of the lines that are not vertices of the polygon, called its \textit{adjoint curve}. 
\begin{figure}[h]
	\begin{center}
	\includegraphics*[scale=0.18]{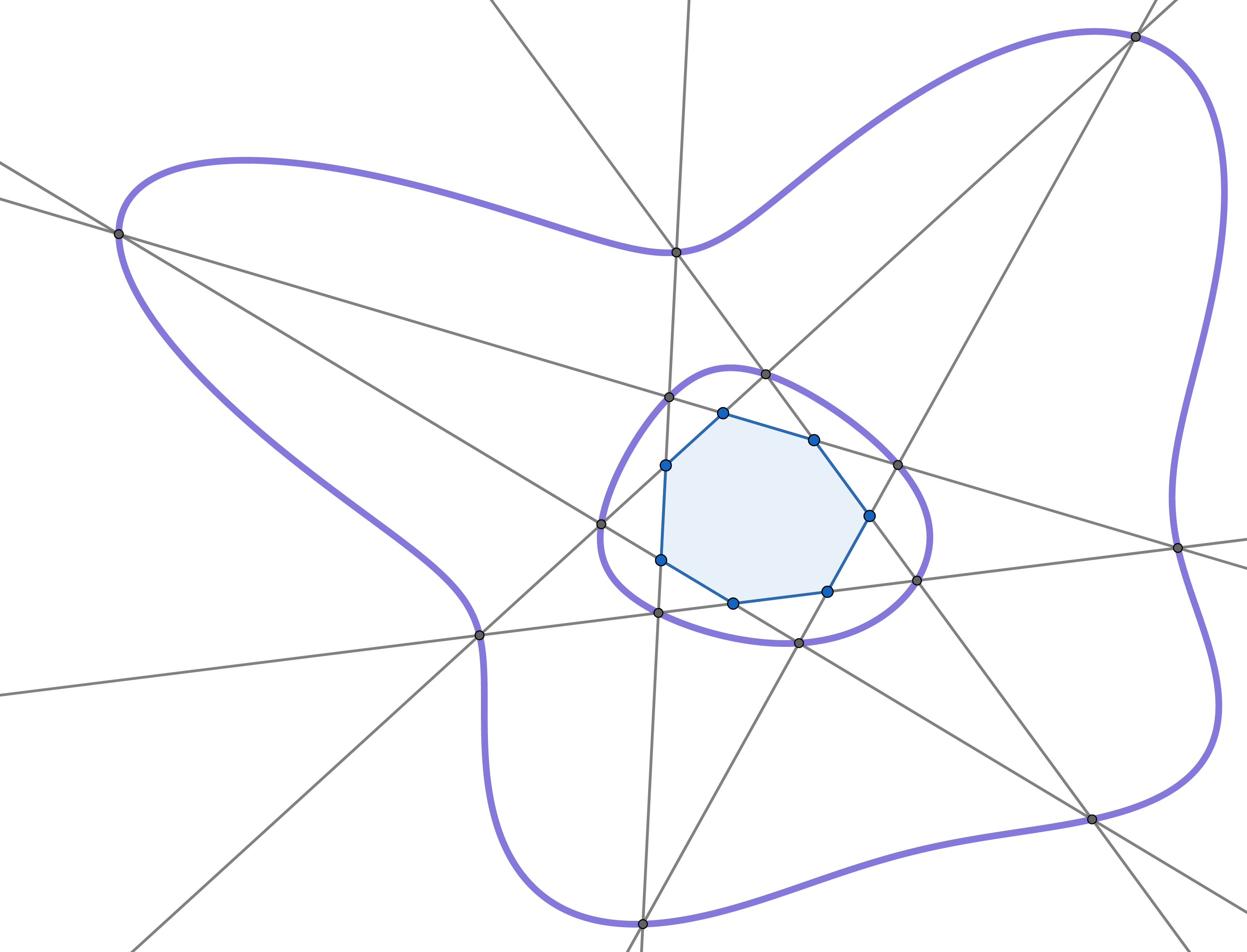}
	\caption{Seven lines with the seven vertices of the heptagon marked in blue; the other 14 pairwise intersection points of the lines are the residual points on the purple adjoint quartic curve.}
	\label{basic7gon}
	\end{center}
\end{figure}

The adjoint plays a role in generalizations of barycentric coordinates from triangles to polygons ($n>3$). These generalized barycentric coordinates in terms of rational functions appear in work of Wachspress \cite{zbMATH06505530} and can be further generalized to higher dimensions, see \cite{zbMATH00989549} and \cite{KRzbMATH07286075}. Adjoints are also the corner stone of recent developments in physics: the adjoint is the numerator of the \emph{canonical form} for polytopes viewed as positive geometries. Polytopes are the fundamental examples of positive geometries which arose in the context of scattering amplitudes in particle physics \cites{zbMATH06854036,polypols}. 

Our focus here is on $n$-gons in the plane. The definition of the adjoint extends naturally to tuples of lines in the complex projective plane. For $3\leq n \leq 6$, there are infinitely many polygons with the same adjoint curve. For $n>7$, a general curve of degree $n-3$ is not adjoint to any $n$-gon. The special case $n=7$ is the only case where there is a finite number of polygons adjoint to a general curve of degree $4$ (see \cite{polypols}*{Theorem 4.2}). The goal of our paper is to compute the cardinality of the generic fibre of the map that associates a $7$-gon in the plane with its adjoint quartic plane curve. 

A $7$-gon or heptagon is an ordered tuple of seven lines $(L_1,\ldots,L_7)$ in the complex projective plane with the property that no three lines meet at a point. The adjoint quartic is the unique curve of degree $4$ containing all \emph{residual points} which are the $14$ intersection points of lines $L_i\cap L_j$ with non-consecutive indices (modulo $7$). A numerical computation in \cite{polypols} suggests that the generic number of such heptagons with the same adjoint should be equal to $864$, up to the action of the dihedral group on the labels of the lines. We prove this numerical result rigorously using enumerative results on curves. The total number $864$ is in fact $864 = 24 \times 36$ -- there are 24 heptagons for each of the $36$ even theta characteristics of the given plane quartic. In particular, the Scorza correspondence will play a crucial role. Our proof avoids issues of excess intersection by providing the upper bound of $864$ using intersection theory on products of curves. The lower bound is based on an explicit analysis of the fibre of the Klein quartic. To show that this fibre is reduced and contains exactly $864$ heptagons, we use the automorphsim group to construct all points in the fibre.

The paper is organized as follows. In \Cref{sec:adjoints}, we set up notation and show that every heptagon defines an adjoint quartic curve together with an even theta characteristic. We also give an outline of the proof in terms of upper and lower bounds in \Cref{lemma:twoconditions}. \Cref{sec:scorza} relates heptagons to the Scorza correspondence. We exploit this connection in \Cref{sec:intersection} to compute the upper bound using intersection theory and the class of the Scorza correspondence. Finally, \Cref{sec:kleinquartic} studies the Klein quartic in detail from the point of view of heptagons. We determine its fibre completely, using symbolic computation.

\section{Adjoints of heptagons}\label{sec:adjoints}

Consider a heptagon $P$ in the projective plane $\mathbb{P}^2$, that is to say seven distinct labelled lines $L_1,\dots,L_7$ such that no three of them meet in the same point. We denote the $\binom{7}{2} = 21$ intersection points as $\{p_{ij}\} = L_i\cap L_j$. From these points, we remove $p_{12},p_{23},p_{34},p_{45},p_{56},p_{67},p_{17}$, the intersection points of consecutive lines (modulo $7$), which we call \textit{vertices}. (Like the vertices of convex heptagon in the real plane with consecutively labelled sides; see \Cref{basic7gon}.) We are left with the $14$ points 
\[ \mathcal{R}(P) = \{p_{13},p_{14},p_{15},p_{16},p_{24},p_{25},p_{26},p_{27},p_{35},p_{36},p_{37},p_{46},p_{47},p_{57} \}. \]
This set is called the \emph{residual arrangement} of the heptagon, cf.~\cite{KRzbMATH07286075} for a modern discussion and \cite{zbMATH06505530} for history.
There exists a unique smooth quartic curve $C$ passing through the residual arrangement, called the adjoint quartic of the heptagon. This defines a rational map from the set of seven ordered lines in $\mathbb{P}^2$ to plane quartics
\begin{equation}\label{eq:mapquartic}
\operatorname{adj}\colon |\mathcal{O}_{\mathbb{P}^2}(H)|^{7} = |\mathcal{O}_{\mathbb{P}^2}(H)| \times \dots \times |\mathcal{O}_{\mathbb{P}^2}(H)| \dashrightarrow |\mathcal{O}_{\mathbb{P}^2}(4H)| 
\end{equation}
where we denote by $H$ the class of a line in $\mathbb{P}^2$. These two spaces have the same dimension and numerical experiments in \cite{polypols} suggest that this map should be generically finite of degree $14\times 864$, where the $14$ comes from the symmetries of the dihedral group $D_7$ acting on the labels of the lines. Concretely, as a subgroup of $S_7$, the dihedral group $D_7$ with $14$ elements is generated by a rotation $(1\ 2\ 3\ 4\ 5\ 6\ 7)$ and a reflection $(2\ 7)(3\ 6) (4\ 5)$.
We prove this conjecture using intersection theory on products of (quartic) curves.

\subsection{Heptagons and theta characteristics}
We first discuss how such a configuration of seven lines determines not only an adjoint quartic but also an even theta characteristic on that quartic. Recall that a theta characteristic of a smooth curve $C$ is a divisor class $\eta$ such that $2\eta \sim K_C$ is the canonical class. It is \textit{even} if the dimension $h^0(C,\eta)$ of the space of global sections is even and it is \emph{odd} otherwise. The difference of two theta characteristics is a $2$-torsion point in the Jacobian $\operatorname{Pic}^0(C)$ of $C$. Since smooth plane quartics are canonical curves, a theta characteristic on such a curve $C\subset \P^2$ is a divisor $\eta$ of degree~$2$ such that $2\eta \sim H$ corresponds to the class of the intersection of $C$ with a line.

\begin{lemma}\label{lem:thetachar}
  Let $P$ be a heptagon with smooth adjoint quartic $C\subset \P^2$. 
  The divisor $R=\sum_{p\in \mathcal{R}(P)} p$ of degree $14$ on $C$ is linearly equivalent to $3H + \eta$, where $\eta$ is a theta characteristic of $C$ with $h^0(C,\eta)=0$. In particular, $\eta$ is even.
\end{lemma}

\begin{proof}
The adjoint quartic $C$ intersects the reducible curve $\sum_{i=1}^7 L_i$ at the 14 points of the residual arrangement $\mathcal{R}(P)$, with multiplicity at least $2$ at each point. Since $7\cdot 4 = 28$, these are all the intersection points and the multiplicity is exactly $2$. 
Hence, looking at the corresponding divisors on the curve $C$, we see that $2\cdot R \sim 7H$, or equivalently
\[ R \sim 3H + \eta \]
where $\eta$ is a theta characteristic on $C$. We now show that this is an even theta characteristic. Consider the two lines $L_1,L_7$ and the subset $\{p_{26},p_{24},p_{25},p_{46},p_{36},p_{35}\}$ of points in $\mathcal{R}(P)$ that are not contained in $L_1$ or $L_7$. Then we have
\begin{equation}\label{eq:ztheta}
 p_{26}+p_{24}+p_{25}+p_{46}+p_{36}+p_{35} \sim R - C.L_1 - C.L_7 \sim H+\eta 
\end{equation}
so that
\begin{equation}\label{eq:scorzafirst} 
p_{35}+p_{36}+p_{46}-p_{27} = p_{26}+p_{24}+p_{25}+p_{46}+p_{36} - C.L_2 \sim \eta. 
\end{equation}
Since the points $p_{35},p_{36},p_{46}$ are not collinear, we have that $h^0(C,p_{35}+p_{36}+p_{46})=1$, and since $p_{27}$ is distinct from all these points, we have that $h^0(C,\eta)=0$.
\end{proof}

We can give a different proof of a more general statement as follows. Recall from \cite{polypols} that a \emph{rational polypol} in the plane is a reducible curve $P=\Gamma_1+\Gamma_2+\dots+\Gamma_r$ of degree $d$ whose irreducible components $\Gamma_i$ are reduced rational curves  with a choice of points $p_{i,{i+1}}\in \Gamma_i\cap \Gamma_{i+1}$ (modulo $r$) that are nodal singularities of $P$. These points are called the vertices of the polypol. The set $\mathcal{R}(P)$ consisting of all the singularities of $P$, except the vertices, is the set of residual points. Assume now that all the singularities of $P$ are nodal, then it is proven in \cite{polypols}*{Proposition 2.2}, that there exists a unique plane curve $C$ of degree $d-3$ passing through the residual points of $P$. This is called the adjoint curve of the polypol. 

\begin{proposition}\label{prop:thetacharpolypol}
	Assume that the rational polypol $P$ is nodal with a smooth adjoint curve $C$. The divisor $R = \sum_{p\in \mathcal{R}(P)} p$ on $C$ is linearly equivalent to $3H+\eta$, where $\eta$ is a theta characteristic of $C$ with $h^0(C,\eta)=0$. In particular, it is even.
\end{proposition}
\begin{proof}
	The curve $C$ intersects $P$ in all the points of $R$, with multiplicity at least $2$ at each point. However, the computation of \cite{polypols}*{Example 2.3} shows that $\mathcal{R}(P)$ consists of $\frac{d(d-3)}{2}$ points. Hence, by B\'ezout's Theorem, the curves $C$ and $P$ intersect only at the points of $R$, with multiplicity exactly $2$ at each point. This means that $2\cdot R \sim dH$, so that $2\cdot (R-3H) \sim (d-6)H \sim K_C$, where the last equivalence follows from the adjunction formula. 
	This proves that $\eta = R-3H$ is a theta characteristic. In particular, $\eta \sim K_C-\eta \sim (d-3)H-R$, and we have an exact sequence of sheaves on $\mathbb{P}^2$:
	\begin{equation}\label{eq:idealsheafR} 
	0 \longrightarrow \mathcal{O}_{\mathbb{P}^2} \longrightarrow \mathcal{I}_{R,\mathbb{P}^2}((d-3)H) \longrightarrow \mathcal{O}_{C}((d-3)H-R) \longrightarrow 0
	\end{equation}
	Taking the long exact sequence in cohomology and using that $H^1(\mathbb{P}^2,\mathcal{O}_{\mathbb{P}^2})=0$, we see that $h^0(C,\eta) = h^0(\mathbb{P}^2,\mathcal{I}_{R,\mathbb{P}^2}((d-3)H))-1$, but this is zero because $C$ is the unique curve of degree $d-3$ passing through the points of $R$.
\end{proof}

Going back to heptagons, Lemma \ref{lem:thetachar} tells us that the previous map ${\rm adj}$ in \eqref{eq:mapquartic} from the set of seven lines factors through a map ${\rm adj}_+$ to the space $|\mathcal{O}_{\mathbb{P}^2}(4H)|^{+}$ of pairs $(C,\eta)$, where $C$ is a smooth quartic curve and $\eta$ is an even theta characteristic on $C$:
\begin{equation}\label{eq:maptheta} 
\adj_+\colon |\mathcal{O}_{\mathbb{P}^2}(H)|^{7} \dashrightarrow |\mathcal{O}_{\mathbb{P}^2}(4H)|^{+}.
\end{equation}
Since each smooth curve of genus $3$ has exactly $36$ distinct even theta characteristics \cite{BLAbvar}*{Proposition 11.2.7}, the natural map $|\mathcal{O}_{\mathbb{P}^2}(4H)|^{+} \to |\mathcal{O}_{\mathbb{P}^2}(4H)|$, which forgets the theta characteristic, has degree $36$. Hence our original map ${\rm adj}$ in \eqref{eq:mapquartic} has degree $14\times 864$ if and only if the map ${\rm adj}_+$ in \eqref{eq:maptheta} has degree $14\times \frac{864}{36} = 14 \times 24 = 336$. We will prove this in two steps, according to the following result.
\begin{lemma}\label{lemma:twoconditions}
	Assume that the following two statements are true:
	\begin{enumerate}
	  \item[$(*)$] Any pair $(C,\eta)$ of a smooth plane quartic $C$ and an even theta characteristic $\eta$ arises from at most $336$ heptagons such that the differential of the map $\operatorname{adj}$ is nonsingular at each of these $336$ heptagons.
	  \item[$(**)$] There exists one pair $(C,\eta)$ of a smooth plane quartic $C$ and an even theta characteristic $\eta$ that has $336$ associated heptagons and such that the differential of the map $\operatorname{adj}$ is nonsingular at each of these $336$ heptagons.
	\end{enumerate}
    Then a general pair $(C,\eta)$ arises from precisely $14\times 24=336$ heptagons, while a general quartic curve $C$ arises from $14\times 864$ heptagons.  
\end{lemma}
\begin{proof}
	First consider the set of heptagons $\mathcal{H}\subseteq |\mathcal{O}_{\mathbb{P}^2}(H)|^7$. We recall that this is the set of seven ordered lines such that no three of them meet in one point. This is open in $|\mathcal{O}_{\mathbb{P}^2}(H)|^7$. The adjoint map is well-defined on $\mathcal{H}$, so that we get a morphism $\operatorname{adj}\colon \mathcal{H} \to |\mathcal{O}_{\mathbb{P}^2}(4H)|$, where we just consider the adjoint and, for now, not a theta characteristic. Let now $\mathcal{U}$ be the preimage of the open subset of smooth quartics. This is again open, and for any heptagon in $\mathcal{U}$ we obtain an even theta characteristic, as explained in Lemma~\ref{lem:thetachar}. Thus we have a morphism $\operatorname{adj}\colon \mathcal{U} \to |\mathcal{O}_{\mathbb{P}^2}(4H)|^+$. The open subset $\mathcal{U}_{\text{sm}}$ where the map has nonsingular differental is exactly the locus where the map is smooth, because of \cite{Hart}*{Proposition 10.4}. In particular it is open and it is nonempy thanks to Assumption $(**)$. Since the restriction $\operatorname{adj}\colon \mathcal{U}_{\text{sm}} \to |\mathcal{O}_{\mathbb{P}^2}(4H)|^+$ is smooth of relative dimension $0$, it is \'etale, and Assumption $(*)$, together with \cite{EGAIV}*{Proposition 15.5.1} (see also \cite{MORomagny}), show that the locus $\mathcal{V}$ of pairs $(C,\eta)$ such that $\operatorname{adj}_{|\mathcal{U}_{sm}}^{-1}(C,\eta)$ consists of $336$ points is open. In turn, Assumption $(**)$ shows that it is nonempty, hence  $\mathcal{V}$ is open and dense in $|\mathcal{O}_{\mathbb{P}^2}(4H)|^+$. At this point, let $\mathcal{V}_o = \mathcal{V}\cap \overline{\operatorname{adj}(\mathcal{U}\setminus \mathcal{U}_{sm})}$: this is open and dense, since $\dim \overline{\operatorname{adj}(\mathcal{U}\setminus \mathcal{U}_{sm})} \leq \dim(\mathcal{U}\setminus \mathcal{U}_{sm}) < \dim \mathcal{U} = \dim \mathcal{V}$. The previous argument shows that any pair $(C,\eta)$ in $\mathcal{V}_o$ arises from precisely $336$ heptagons, counted with multiplicity, and the discussion before this lemma shows that any curve in $C$ arises from $14\times 864$ heptagons, counted with multiplicity. 
\end{proof}

We obtain the bound in $(*)$ in Section \ref{sec:scorza} via intersection theory. The example we give for Statement $(**)$ is the Klein quartic in Section \ref{sec:kleinquartic}.

\begin{remark}\label{rmk:irreducibility}
In the previous argument, we also need the space $|\mathcal{O}_{\mathbb{P}^2}(4H)|^{+}$ to be irreducible. It is well-known that an even theta characteristic on a smooth plane curve corresponds to a symmetric determinantal representation of the curve \cite{BeaDet}. Hence the space $|\mathcal{O}_{\mathbb{P}^2}(4H)|^{+}$ has a parametrization from an open subset of the space of symmetric matrices of linear ternary forms, which is clearly irreducible.
\end{remark} 

\section{Heptagons and Scorza correspondences}\label{sec:scorza} 

We will now consider the fibers of the map \eqref{eq:maptheta}: we want to reconstruct a heptagon from a smooth plane quartic and a theta characteristic. We will do so via the Scorza correspondence, inspired by a construction of Mukai in \cite{Mukai} to which we refer for more details. See also \cite{Dolgachev}*{Section 5.5}. Recall that if $C$ is a smooth plane quartic and $\eta$ is an even theta characteristic, the classical \emph{Scorza correspondence} is the subset of $C\times C$ defined as
\[ S(\eta) = \{ (x,y) \in C\times C \,|\, h^0(C,\eta+x-y)>0 \} \]
This is a divisor in $C\times C$. Since $h^0(C,\eta)=0$, it follows that $S(\eta)\cap \Delta = \emptyset$. Furthermore, by Riemann-Roch, $h^0(\eta+x-y)=h^0(\eta+y-x)$ so that the Scorza correspondence is symmetric with respect to exchanging the factors.

Suppose now that $(C,\eta)$ is adjoint to a heptagon $P=(L_1,\dots,L_7)$. We distinguish the residual points into ``inner'' and ``outer'' residual points:
\begin{align*}
\text{Inner  points: } & \qquad p_{13}, p_{24}, p_{35}, p_{46}, p_{57}, p_{16}, p_{27}. \\
\text{Outer  points: } & \qquad p_{14}, p_{25}, p_{36}, p_{47}, p_{15}, p_{26}, p_{37}.
\end{align*}
The inner residual points are the intersection points of lines whose indices are two apart (in cyclic order), and the outer residual points the intersection points of lines whose indices are three apart. The name is inspired by the real convex case shown in \Cref{basic7gon} where the inner residual points are the seven points on the inner oval of the (hyperbolic) adjoint quartic curve and the outer residual points the seven points on the outer oval. Here, the labels of the seven lines come from order of the edges of the convex heptagon in clockwise (or counterclockwise) order. 

Observe that the heptagon can be reconstructed from the set of labeled inner residual points. More precisely, let $\mathcal{H}\subseteq |\mathcal{O}_{\mathbb{P}^2}(H)|^7$ be the open set of heptagons (no three lines meeting in one point) and define the map
\begin{equation}\label{eq:matoinnerpoints} 
I \colon \mathcal{H} \longrightarrow (\mathbb{P}^2)^7, \quad (L_1,\dots,L_7) \mapsto (
p_{13},p_{46},p_{27},p_{35},p_{16},p_{24},p_{57})  
\end{equation}
that associates to a heptagon the set of labeled inner residual points. Then the inverse of this map is given by 
\begin{equation}\label{eq:maptoheptagons}
J\colon (\mathbb{P}^2)^7 \dashrightarrow |\mathcal{O}_{\mathbb{P}^2}(H)|^7, \quad (x_1,\dots,x_7) \mapsto (\ell_{15},\ell_{36},\ell_{14},\ell_{26},\ell_{47},\ell_{25},\ell_{37})
\end{equation}
that associates to a set of seven distinct inner residual points the corresponding heptagon, where we write $\ell_{ij}$ for the line spanned by $x_i$ and $x_j$. Therefore, the problem of counting heptagons can be translated into counting the configurations of inner residual points. To do so, we analyze these configurations as follows.

\begin{lemma}\label{lemma:scorzaall} 
	With the previous notation, we have the following equivalences of divisors:
	\begin{align*}
	\eta+p_{13}-p_{46} &\sim p_{47} + p_{57}, & \eta+p_{46}-p_{13} &\sim p_{37} + p_{27}, \\
	\eta+p_{46}-p_{27} &\sim p_{37} + p_{13}, & \eta+p_{27}-p_{46} &\sim p_{36} + p_{35}, \\
	\eta+p_{27}-p_{35} &\sim p_{36} + p_{46}, & \eta+p_{35}-p_{27} &\sim p_{26} + p_{16}, \\
	\eta+p_{35}-p_{16} &\sim p_{26} + p_{27}, & \eta+p_{16}-p_{35} &\sim p_{25} + p_{24}, \\
	\eta+p_{16}-p_{24} &\sim p_{25} + p_{35}, & \eta+p_{24}-p_{16} &\sim p_{15} + p_{57}, \\
	\eta+p_{24}-p_{57} &\sim p_{15} + p_{16}, & \eta+p_{57}-p_{24} &\sim p_{14} + p_{13}, \\
	\eta+p_{57}-p_{13} &\sim p_{14} + p_{24}, & \eta+p_{13}-p_{57} &\sim p_{47} + p_{46}. 
	\end{align*}
\end{lemma}
\begin{proof}
	We showed in Equation \eqref{eq:scorzafirst} that $\eta \sim p_{35}+p_{36}+p_{46}-p_{27}$, so that $\eta+p_{27}-p_{35} \sim p_{36}+p_{46}$.  Now observe that $p_{36}+p_{46}+p_{26}+p_{16} \sim C\cap L_6 \sim H$, so that $p_{26}+p_{16}\sim H-(p_{36}+p_{46}) \sim H-(\eta+p_{27}-p_{35}) \sim \eta+p_{35}-p_{27}$.  This proves the third row of the above table. For the others, we can either repeat the reasoning behind Equation \eqref{eq:scorzafirst} or we can use the symmetries with respect to the dihedral group, acting on the labels via cyclic permutations. 
\end{proof} 

Next, we define the divisors $S_{ij} = S_{ij}(\eta) = \pr^*_{ij}(S(\eta))$ in the Cartesian product $C^7$ as the pullbacks of the Scorza correspondence (for $1\leq i<j\leq 7$). We also set
\[ T_7 = S_{12}\cap S_{23} \cap S_{34} \cap S_{45} \cap S_{56} \cap S_{67} \cap S_{17}. \]

\begin{proposition}\label{prop:reconstructionscorza}
	Let $C$ be a smooth plane quartic and $\eta$ an even theta characteristic on $C$. Let $\adj^{-1}(C,\eta)$ be the set of heptagons giving rise to $(C,\eta)$, then the map
	\[ I\colon \adj^{-1}(C,\eta) \longrightarrow C^7 \quad (L_1,\dots,L_7) \mapsto (p_{13},p_{46},p_{27},p_{35},p_{16},p_{24},p_{57}) \]
	induces an isomorphism between $\adj^{-1}(C,\eta)$ and the set $T_7\cap J^{-1}(\mathcal{H})$.
	In particular, Statement $(*)$ in Lemma \ref{lemma:twoconditions} is true if $T_7\cap J^{-1}(\mathcal{H})$ contains at most $336$ isolated points.
\end{proposition}
\begin{proof}
	By definition, the image $I(\adj^{-1}(C,\eta))$ under the map $I$ that takes seven lines to the inner residual points is contained in $J^{-1}(\mathcal{H})=I(\mathcal{H})$. Lemma \ref{lemma:scorzaall} shows that it is also contained in $T_7$. So $I$ induces a map $I\colon \adj^{-1}(C,\eta) \to T_7\cap J^{-1}(\mathcal{H})$. To prove that this is an isomorphism we need to show that $J(T_7\cap J^{-1}(\mathcal{H})) \subseteq \adj^{-1}(C,\eta)$, meaning that if we have a configuration of points $(x_1,\dots,x_7) \in T_7\cap J^{-1}(\mathcal{H})$, the curve $C$ and the theta characteristic $\eta$ are adjoint to the heptagon $J(x_1,\dots,x_7) = (\ell_{15},\ell_{36},\ell_{14},\ell_{26},\ell_{47},\ell_{25},\ell_{37})$. To do so, we need to prove that all pairs of lines in the heptagon, apart from consecutive lines, meet in a point that lies on the curve $C$.
	
	Since $(x_1,\dots,x_7)\in T_7$, we know $\eta+x_1-x_2 \sim A$ and $\eta+x_1-x_7\sim B$ for two effective divisors $A,B$ of degree two. In particular, $B+x_7 \sim \eta+x_1 \sim A+x_2$. Since $h^0(C,\eta)=0$, it follows that $h^0(C,\eta+x_1)=1$. So two linearly equivalent divisors in $|\eta+x_1|$ are actually equal, so that $B+x_7 = A+x_2$. Since $J(x_1,\dots,x_7)$ is a heptagon, it must be that $x_2\ne x_7$. Therefore, $A = y_1+x_7$ and $B=y_1+x_2$ for a certain $y_1\in C$. With a similar reasoning, or by acting on the labels via cyclic permutations, we see that there are $y_i\in C,i=1,\dots,7$ such that
	\begin{align*}
	\eta+x_1-x_2 &\sim y_1 + x_7, & \eta+x_2-x_1 &\sim y_2 + x_3, \\
	\eta+x_2-x_3 &\sim y_2 + x_1, & \eta+x_3-x_2 &\sim y_3 + x_4, \\
	\eta+x_3-x_4 &\sim y_3 + x_2, & \eta+x_4-x_3 &\sim y_4 + x_5, \\
	\eta+x_4-x_5 &\sim y_4 + x_3, & \eta+x_5-x_4 &\sim y_5 + x_6, \\
	\eta+x_5-x_6 &\sim y_5 + x_4, & \eta+x_6-x_5 &\sim y_6 + x_7, \\
	\eta+x_6-x_7 &\sim y_6 + x_5, & \eta+x_7-x_6 &\sim y_7 + x_1, \\
	\eta+x_7-x_1 &\sim y_7 + x_6, & \eta+x_1-x_7 &\sim y_1 + x_2. 
	\end{align*}
	 Let us now consider the heptagon $(\ell_{15},\ell_{36},\ell_{14},\ell_{26},\ell_{47},\ell_{25},\ell_{37})$. By construction, the lines that are two steps apart, such as $\ell_{15}$ and $\ell_{14}$, meet on $C$. Let us show that the lines that are three steps apart, such as $\ell_{15}$ and $\ell_{26}$, also meet on $C$. Indeed, summing the last two rows in the previous table, we see that $x_1+x_5+y_6+y_7 \sim H$ and $x_2+x_6+y_1+y_7\sim H$. This shows that $\ell_{15}$ and $\ell_{26}$ meet at the point $y_7\in C$. For the other pairs of lines that are three steps apart, we can either repeat this reasoning or use the action of the cyclic permutations on the labels.
	 
	 Finally, if $(L_1,\dots,L_7) \in \adj^{-1}(C,\eta)$ is a heptagon where the adjoint map has nonsingular differential, it must be an isolated point in $\adj^{-1}(C,\eta)$ and hence corresponds to an isolated point of $T_7\cap J^{-1}(\mathcal{H})$. This proves the last statement.   
\end{proof}

Our task now is to count the number of isolated points in $T_7\cap J^{-1}(\mathcal{H})$. We will do so via intersection theory, following Mukai \cite{Mukai}.

\section{Intersection calculus on products of curves}\label{sec:intersection}
We will compute the number of points in $T_7$ via intersection products on the space $C^7$. First, we recall some notions on the intersection product on Cartesian products.

\subsection{Cartesian Products of Curves}
Let $C$ be a smooth complete curve of genus $g$. We collect some results on the intersection product on the cartesian product $C^n$, using \cite{ACGH} and \cite{Bastetal} as our main references. For indices $1\leq i_1< i_2< \dots< i_k\leq n$, we denote by \[\operatorname{pr}_{i_1,\ldots,i_k}\colon C^n \to C^k\] the projection onto the components indexed by $i_1,\dots,i_k$. We will consider the ring $N(C^n)$ of cycles up to numerical equivalence. In particular, we will focus on the subring of $N(C^n)$ generated by the following two types of divisors (usually called \emph{tautological ring of cycles}). We write $x_i$ for the class $\pr_i^*(x)$, where $1\leq i \leq n$ and $x$ is any point of $C$, and $\Delta_{ij}$ for the class $\operatorname{pr}_{ij}^*\Delta$, where $\Delta\subseteq C\times C$ is the diagonal. We furthermore denote by $\Delta_{12\dots n}$ the small diagonal $\Delta = \{(x,x,\dots,x) \,|\, x\in C \} \subseteq C^n$, which is also the intersection of the divisors $\Delta_{12},\Delta_{23},\ldots,\Delta_{n-1,n}$.
The next result collects the relations among these classes.
\begin{lemma}\label{lemma:basicrelations}
	Let $C$ be a smooth curve of genus $g$. The following equalities hold in $N(C^n)$.
	\begin{enumerate}[label=(\arabic*)]
		\item $x_1\cdot x_2 \cdot \dots \cdot x_n = 1$.
		\item $x_i^2 = 0$,  $x_i\cdot \Delta_{ij} = x_i\cdot x_j$  and $\Delta_{ij}^2 = -(2g-2)\cdot x_i\cdot x_j$.
		\item $\Delta_{12\dots n} \cdot \Delta_{ij} = -(2g-2)$. 
	\end{enumerate}
\end{lemma}
\begin{proof}
	(1) The divisors $\pr_i^*(x_i)$ for $i=1,\ldots,n$ intersect transversally in the unique point $(x_1,x_2,\dots,x_n)$. 
	
	(2) This is because the first relation holds in $C$ and the last two hold in $C\times C$. We have $x_1\cdot x_2=1$ because of (1) and $x_i \cdot \Delta = 1$ because the divisors $\pr_i^*(x_i)$ and $\Delta$ intersect transversally in the unique point $\{(x_i,x_i)\}$. To show that $\Delta^2 = -(2g-2)$ holds on $C\times C$, consider the diagonal embedding $j\colon C \to C^2, x\mapsto (x,x)$. By the adjunction formula, we have $j^*\Delta \sim K_C - j^*(K_{C^2})$. Since $j^*(K_{C^2}) = j^*(K_C\boxtimes K_C) = 2K_C$ we conclude $j^*\Delta \sim -K_C$, which has degree $-(2g-2)$.
	
	(3) The small diagonal is the image of the embedding $j\colon C \to C^n, x\mapsto (x,x,\dots,x)$. By definition, the intersection number is $\deg j^*\Delta_{ij} = \deg j^*\operatorname{pr}_{ij}^*\Delta$, so we can just look at the composed map $\operatorname{pr}_{ij}\circ j$, which is the usual diagonal embedding in $C\times C$. So the intersection number is the same as $\Delta^2$ on $C\times C$, which is $-(2g-2)$. 
\end{proof} 

In general, the relations in Lemma \ref{lemma:basicrelations} show that any monomial in the classes $x_i$ and $\Delta_{ij}$ can be written in reduced, square free form
\begin{equation}\label{eq:monomialform} 
x_{i_1}\cdot x_{i_2} \cdot \dots \cdot x_{i_r} \cdot \Delta_{h_1k_1} \cdot \dots \Delta_{h_sk_s} 
\end{equation}
where $i_1,\dots,i_r$ and $(h_1k_1),\dots,(h_sk_s)$ are pairwise distinct and the two groups of indices are disjoint, i.e.~
\[ \{ i_1,\dots,i_r \} \cap \{ h_1,k_1,\dotsc,h_s,k_s \} = \emptyset. \]
To determine the degree of this class, assuming it is $0$-dimensional, it is enough to compute $\Delta_{h_1k_1} \cdot \dotsc \cdot \Delta_{h_sk_s}$. 
\begin{lemma}\label{lemma:graphrelation}
	To compute $\Delta_{h_1k_1} \cdot \dotsc \cdot \Delta_{h_sk_s}$, define the graph $\Gamma$ on the vertex set $V=\{h_1,k_1,\dots,h_s,k_s \}$ where two vertices are connected by an edge if and only if they appear as one of the indexes in $\Delta_{h_1k_1} \cdot \dotsc \cdot\Delta_{h_sk_s}$. 
	Assume that this graph $\Gamma$ is connected.
	\begin{enumerate}[label=(\arabic*)]
		\item The graph $\Gamma$ has precisely $s$ edges.
		\item If $\Gamma$ has more than $s$ vertices, then it is a tree on $s+1$ vertices and $\Delta_{h_1k_1}\cdot \dotsc \cdot \Delta_{h_sk_s}$ is the class of a small diagonal in $C^{s+1}$.
		\item If $\Gamma$ has $s$ vertices, then $\Delta_{h_1k_1}\cdot \dotsc \cdot \Delta_{h_sk_s} = -(2g-2) \in N(C^s)$.
		\item If $\Gamma$ has less than $s$ vertices, then $\Delta_{h_1k_1}\cdot \dotsc \cdot \Delta_{h_sk_s} = 0$.
		\item If $\Gamma$ is not connected, then the product $\Delta_{h_1k_1} \cdot \dotsc \cdot \Delta_{h_sk_s}$ can be computed as a product over the different connected components of $\Gamma$
	\end{enumerate}
    
\end{lemma}
\begin{proof}
	(1) This is clear by construction.
	
	(2)  If $\Gamma$ is connected but has more than $s$ vertices, then it is a tree with $s+1$ vertices. We show the claim by induction on $s$. In the base case $s=1$, we have the usual diagonal in $C^2$. The induction step of adding one leaf and the corresponding edge to the tree is \Cref{lemma:basicrelations}(3).
	
	(3) If $\Gamma$ has $s$ vertices and $s$ edges, then it has a spanning tree with $s-1$ edges. Up to relabeling, we may assume that this corresponds to $\Delta_{h_1,k_1}\cdot \dotsc \cdot \Delta_{h_{s-1}k_{s-1}}$. By (1), this is the class of a small diagonal in $C^s$, and the intersection with the last diagonal $\Delta_{h_s,k_s}$ is computed in Lemma \ref{lemma:basicrelations} to be $-(2g-2)$.
	
	(4) In this case, we are intersecting $s$ divisor classes in the space $C^k$ for some $k<s$, so we obtain $0$.
	
	(5) This is true by construction of the graph $\Gamma$. 
\end{proof}

\subsection{The class of the Scorza correspondence}

Let now $C$ be a smooth plane quartic curve and $\eta$ an even theta characteristic such that $h^0(C,\eta)=0$. We consider the Scorza correspondence $S(\eta)\subseteq C\times C$.  According to \cite{Dolgachev}*{Equation (5.37)}, this is linearly equivalent to  $S(\eta) \sim \pr_1^*\eta + \pr_2^*\eta + \Delta$, hence its class in $N(C^2)$ is
\begin{equation}\label{eq:scorzaclass} 
[S(\eta)] = 2x_1+2x_2+\Delta. 
\end{equation}

\begin{lemma}\label{lemma:scorzanef}
	The class of the Scorza correspondence is nef, meaning that $S(\eta) \cdot D \geq 0$ for every curve $D\subseteq  C^2$.
\end{lemma}
\begin{proof}
	Let $D\subseteq C^2$ be an irreducible curve. If $D$ is not contained in $S(\eta)$, then $S(\eta)\cdot D \geq 0$. If instead $D\subseteq S(\eta)$, then in particular $D\cap \Delta = \emptyset$, thus
	\[ D\cdot S(\eta) = D\cdot (2x_1+2x_2+\Delta) = 2D\cdot (x_1+x_2) \]
	The class $x_1+x_2$ is ample on $C\times C$ because it is the class of the ample line bundle $\pr_1^*\mathcal{O}_C(x_1)\otimes \pr_2^*\mathcal{O}_C(x_2)$. Hence $D\cdot (x_1+x_2) >0$ and $D\cdot S(\eta) >0$, as well. 
\end{proof}

Now we consider on $C^n$, as before, the pullbacks $S_{ij} = \pr_{ij}^*S(\eta)$ and the intersection $T_n = S_{12}\cap S_{23}\cap \dots \cap S_{(n-1)n} \cap S_{1n}$. This is an intersection of $n$ divisors in a space of dimension $n$, so we expect it to have dimension $0$. In this case we can compute the number of points  via the product of the classes of the $S_{ij}$ in $N(C^n)$. Actually, in Proposition \ref{prop:reconstructionscorza} we just need to bound the number of isolated points in $T_7$, and this can be done even without showing that the whole set is finite because $S(\eta)$ is nef.

\begin{lemma}\label{lemma:scorzabound}
	For every $n\geq 2$, let $S_{123\dots n} = S_{12}\cap S_{23}\cap \dots \cap S_{n-1,n}$, so that $T_n = S_{1234\dots n}\cap S_{1n}$. The following holds:
	\begin{itemize}
		\item[(1)] The intersection $S_{123\dots n}$ is of pure dimension $1$, as expected.
		\item[(2)] The number of isolated points (counted with multiplicity) in $T_n$ is bounded by the intersection product $S_{12}\cdot \ldots \cdot S_{(n-1)n}\cdot S_{1n}$. 
	\end{itemize}
\end{lemma}
\begin{proof}
		(1) We proceed by induction on $n$. If $n=2$, then $S_{12}$ is the Scorza correspondence in $C^2$, which is a curve. Now consider the projection $\pr_{12\ldots n-1}\colon S_{12\ldots n} \to C^{n-1}$ which forgets the last component. The image of this map is contained in $S_{12\dots n-1}$, which has pure dimension $1$ by the induction hypothesis. To conclude, it is enough to show that the map has finite fibers. However, the fiber of this map over $(x_1,\ldots,x_{n-1})$ is identified with the subset $\{x_n \in C \,|\, h^0(C,\eta+x_{n-1}-x_n)>0 \}$. This corresponds to the unique effective divisor in $|\eta+x_{n-1}|$, which has degree three. Recall that $h^0(C,\eta+x_{n-1})=1$ because $h^0(C,\eta)=0$.

		(2) Write the curve $S_{123\ldots n} = D_1\cup \dots \cup D_r$ as a union of irreducible curves. The isolated points of $T_n$ are given by the intersection of $S_{1n}$ with the $D_i$ that are not contained in $S_{1n}$. Up to relabeling, we can assume that these are $D_1,\dots,D_h$.
		If we write the class $[S_{123\dots n}] = m_1D_1+\dots+m_hD_h$, the number of isolated points of $T_n$, counted with multiplicity, is $\sum_{i=1}^h m_i \cdot (D_i\cdot S_{1n})$. Since the class $S_{1n}$ is the pullback of a nef class, it is nef (see \cite{Laz}*{Example 1.4.4.(i)}), so that $S_{1n}\cdot D_i \geq 0$ for all $i=h+1,\dots, r$. Hence the number of isolated points is less than or equal to the intersection product $S_{1234\dots n}\cdot S_{1n}$. Since $S_{123\dots n}$ has the expected dimension, its class is given by $S_{12}\cdot \ldots \cdot S_{(n-1)n}$ and we are done.
\end{proof}

Next, we compute the relevant intersection products.
\begin{lemma}\label{lemma:scorzaclassmany}
	For any $n\geq 2$, the equality
	\[ S_{12}\cdot \ldots \cdot S_{(n-1)n} \cdot S_{1n} = 2\cdot 3^n-6 \]
    holds in $N(C^n)$. In particular, $S_{12}\cdot \dots \cdot S_{67} \cdot S_{17} = 4368$ in $N(C^7)$.
\end{lemma}
\begin{proof}
	Using \eqref{eq:scorzaclass}, we can write
	\[ S_{12}\cdot \ldots \cdot S_{(n-1)n} \cdot S_{1n} = \prod_{i=1}^n (2x_i + 2x_{i+1}+\Delta_{i,i+1}) = \sum a_1\dots a_n \]
	where in the product we consider the indices modulo $n$ and the last sum is over all possible choices of $a_i \in \{2x_i,2x_{i+1},\Delta_{i,i+1}\}$. Consider now a product $a_1\dots a_n$ and assume that it is nonzero. Assume also that $a_1=2x_2$. Since the product is nonzero, we see that $a_2 \ne 2x_2$ so that $a_2=2x_3$ or $a_2 = \Delta_{2,3}$. Hence $a_1a_2 = 2c_2\cdot x_2x_3$, where $c_2=2$ if $a_2=2x_3$ and $c_2 = 1$ if $a_2 = \Delta_{2,3}$. So we can rewrite $a_1\dots a_n = 2\cdot c_2 \cdot  x_2x_3 \cdot a_3\dots a_n$. Again, because this class is nonzero, we must have $a_3\ne 2x_3$ and we can repeat the previous reasoning. This way we see that
	\[  a_1\dots a_n = 2^k (x_1\dots x_n) = 2^k, \qquad \text{ where } k = |\{i \in \{1,\dots, n \} \,|\, a_i\ne \Delta_{i,i+1} \}|. \]
	An analogous reasoning shows that the same holds if $a_1 = 2x_{1}$. With this in mind, fix a subset $I\subseteq \{1,\dots,n\}$ with $I\ne \emptyset$. We claim that, amongst the products $a_1\dots a_n$ such that $a_i \ne \Delta_{i,i+1}$ if and only if $i\in I$, there are precisely two that are nonzero and that these are both equal to $2^{|I|}$. Indeed, up to a cyclic permutation, we may assume that $1\in I$, so that $a_1=2x_1$ or $a_1=2x_2$. If $a_1\dots a_n\ne 0$, the previous computation shows that the other $a_i$ are uniquely determined by the value of $a_1$ and that the product is $a_1\dots a_n = 2^{|I|}$. This proves our claim. 

	We are left to consider the product $\Delta_{12}\cdot \Delta_{23} \cdot \ldots\cdot \Delta_{1n}$. Since the genus of $C$ is $3$, Lemma~\ref{lemma:graphrelation} shows that this is $-4$. In summary 
	\begin{align*}
	S_{12}\cdot S_{23} \cdot \ldots \cdot S_{1n} &= \sum a_1\cdot \ldots \cdot a_n = \sum_{k=1}^n \binom{n}{k} 2\cdot 2^k -4 = 2((2+1)^n-1)-4 = 2\cdot 3^n-6 
    \end{align*}
    For $n=7$, this evaluates to $4368$.  
\end{proof}

\begin{remark}
	This can be also checked by an explicit computation in \texttt{OSCAR}, see \cite{code}.
\end{remark} 

In the setting of Proposition \ref{prop:reconstructionscorza}, Lemma \ref{lemma:scorzaclassmany} together with Lemma \ref{lemma:scorzabound} provides a bound of $4368$ for the number of isolated points in $T_7\cap J^{-1}(\mathcal{H})$. This bound is too high, as we are aiming for $336$. However, we are also counting some points that do not belong to $J^{-1}(\mathcal{H})$. To analyze this, we consider special configurations, which are called \enquote{biscribed triangles} in \cite{Mukai}:

\begin{lemma}\label{lemma:scorzatriangles} 
\begin{itemize}
	\item[(1)] On $C^3$, the set $T_3 = S_{12}\cap S_{23} \cap S_{13}$ consists of $48$ points, counted with multiplicity.
	\item[(2)] For any $[p_1,p_2,p_3]\in T_3/\mathfrak{S}_3$, there are $504$ configurations in $T_7\setminus J^{-1}(\mathcal{H})$. In total, there are $4032$ such configurations in $T_7\setminus J^{-1}(\mathcal{H})$.
\end{itemize}
\end{lemma}
\begin{proof}
		(1) If the set $T_3$ is finite, then its number of points counted with multiplicity can be computed via the intersection product $S_{12}\cdot S_{23}\cdot S_{31} = 48$, where the last equality follows from Lemma \ref{lemma:scorzaclassmany}. This computation was already done by Mukai in \cite{Mukai}*{Proposition p.~4} to count the biscribed triangles to a plane quartic. We include here a proof that the set $T_3$ is indeed finite. If $(p_1,p_2,p_3)\in T_3$, a reasoning similar to that in Proposition \ref{prop:reconstructionscorza} shows that 
		\begin{align*}
		\eta+p_1-p_2 &\sim p_3+q_1, & \eta+p_2-p_1 &\sim p_3+q_2, \\
		\eta+p_2-p_3 &\sim p_1+q_2, & \eta+p_3-p_2 &\sim p_1+q_3, \\
		\eta+p_3-p_1 &\sim p_2+q_3, & \eta+p_1-p_3 &\sim p_2+q_1.
		\end{align*}
		If we sum the rows, we see that the lines $\ell(q_1,q_2),\ell(q_2,q_3),\ell(q_3,q_1)$ are tangent to $C$ at $p_3,p_1,p_2$ respectively. Hence the name biscribed triangles. To show that $T_3$ is finite, we will show that the projection $T_3\to C$ to each component has finite image, or, equivalently, that it is not surjective. By symmetry, it is enough to consider the projection onto the first component. We show that if $(p_1,p_2,p_3)\in T_3$, then $p_1$ cannot be a point such that its tangent line is bitangent to $C$. Assume by contradiction that $p_1$ is such a point, so that $2p_1+2p'_1 \sim H$ for some $p'_1 \in C$. But we know that $q_2+q_3+2p_1 \sim H$ as well, so that $q_2+q_3 \sim 2p'_1$. Since $C$ is not hyperelliptic, any two linearly equivalent effective divisors of degree $2$ are equal, implying $q_2=q_3=p'_1$. This shows $\eta+p_2-p_3 \sim \eta+p_3-p_2$, so that $2p_2\sim 2p_3$. The same reasoning as before shows $p_2=p_3$. But this is impossible, because then $h^0(C,\eta+p_2-p_3) = h^0(C,\eta)=0$.  
		
        (2) Let $(p_1,p_2,p_3)\in T_3$. Following the reasoning in the previous point, we see that
		\begin{align*}
		\eta+p_1 \sim p_2+p_3+q_1, \quad \eta+p_2\sim p_1+p_3+q_2, \quad \eta+p_3\sim p_1+p_2+q_3
		\end{align*}
	    In particular, by Riemann-Roch, or by the symmetry of the Scorza correspondence, we see that $h^0(C,\eta+q_i-p_i)>0$ for $i=1,2,3$, so that
	    \begin{align*}
	    \eta+q_1 \sim a_1+b_1+p_1, \quad \eta+q_2 \sim a_2+b_2+p_2, \quad \eta+q_3 \sim a_3+b_3+p_3 
	    \end{align*}
        for certain $a_i,b_j\in C$. Consider now the following graph:

        \begin{center}
        	\begin{tikzcd}[cramped,sep=small]
        		&                         & a_3 \arrow[rd, no head] &                         & b_3 \arrow[ld, no head] &                         &                        \\
        		&                         &                         & q_3 \arrow[d, no head]  &                         &                         &                        \\
        		&                         &                         & p_3 \arrow[ld, no head] &                         &                         &                        \\
        		&                         & p_1 \arrow[rr, no head] &                         & p_2 \arrow[lu, no head] &                         &                        \\
        		a_1 \arrow[r, no head] & q_1 \arrow[ru, no head] &                         &                         &                         & q_2 \arrow[lu, no head] & a_2 \arrow[l, no head] \\
        		& b_1 \arrow[u, no head]  &                         &                         &                         & b_2 \arrow[u, no head]  &                       
        	\end{tikzcd}
           \end{center}
       We claim that if $x,y\in C$ are points connected by an edge in this graph, then $(x,y)
       \in S(\eta)$. For example, for $q_1,b_1$ we see from the previous linear equivalences that $\eta+q_1-b_1 \sim a_1+p_1$, so that $(q_1,b_1)\in S(\eta)$. Since the Scorza correspondence is symmetric, it follows that $(b_1,q_1)\in S(\eta)$, as well. An analogous reasoning holds for all other edges. In particular, any closed walk of length $7$ along this graph, with vertices $y_1\to y_2\to y_3\to y_4 \to y_5 \to y_6 \to y_7 \to y_1$, yields a configuration $(y_1,y_2,\dots,y_7) \in T_7$. One can compute the number of such closed walks by writing down the $12\times 12$ adjacency matrix of the above graph and then taking the trace of its $7$-th power. This can be implemented on a computer, for example in \texttt{Julia}, and it yields $504$. Furthermore, any permutation of the elements $(p_1,p_2,p_3)\in T_3$ will induce a permutation in the set of closed walks in the graph, so that it gives the same set of configurations. This shows that for every point $(p_1,p_2,p_3)\in T_3$ there are $504/6$ such configurations in $T_7$, so that by (1) we get $48 \cdot 504/6 = 4032$ in total. To conclude, we need to show that these configurations do not belong to $J^{-1}(\mathcal{H})$. This is simply because in any such configuration there are two identical points, so it cannot arise from a heptagon.
\end{proof}

\begin{corollary}\label{cor:conditionone}
	The number of isolated points in $T_7 \cap J^{-1}(\mathcal{H})$ is bounded by $336$ and Statement $(*)$ in Lemma \ref{lemma:twoconditions} is true.
\end{corollary}
\begin{proof}
	Lemma \ref{lemma:scorzabound} and Lemma \ref{lemma:scorzaclassmany} show that the number of isolated points in $T_7$, counted with multiplicity, is bounded by $4368$. Furthermore, Lemma \ref{lemma:scorzatriangles} shows that amongst these isolated points there are $36\cdot 14\cdot 48/6 = 4032$ points that are contained in $T_7 \setminus J^{-1}(\mathcal{H})$. Hence, the number of isolated points in $T_7 \cap J^{-1}(\mathcal{H})$ is bounded by $4368-4032 = 336$. At this point, Proposition \ref{prop:reconstructionscorza} implies Statement $(*)$ of Lemma \ref{lemma:twoconditions}.  
\end{proof}

\section{The Klein Quartic}\label{sec:kleinquartic}

In this section, we prove Statement $(**)$ in Lemma \ref{lemma:twoconditions}, focusing on the Klein quartic. This is the smooth plane curve $C$ given by the equation
\[x^3y + y^3z + z^3 x=0.\]
This classical curve is very well studied, see e.g.~\cite{zbMATH01574981} for a survey. We will show that there is a theta characteristic $\etasym$ on $C$ such that $(C,\eta)$ are adjoint to $336$ heptagons, and that the map $\operatorname{adj}$ is unramified at each of these heptagons. We achieve this by using the large automorphism group $\operatorname{Aut}(C) \cong \mathrm{PSL}_2(\mathbb{F}_7)$ of $168$ elements. Notice that the automorphism group acts linearly on $C$, i.e.~as a subgroup of $\mathrm{PGL}_3(\mathbb{C})$. This is because the plane model of $C$ is precisely its canonical embedding.

We make use of the fact that among the $36$ even theta characteristics of $C$ there is exactly one that is invariant under all automorphisms (see \cite{zbMATH05874806}), which we call $\etasym$. In the plane model, it can be represented as \[\etasym + H \sim 2(e_1 + e_2 + e_3)\] where $e_1 = [1,0,0], e_2=[0,1,0], e_3=[0,0,1]$. The points $e_1$, $e_2$, $e_3$ are inflection points of the Klein quartic: We have $H \sim 3 e_1 + e_3 \sim 3 e_3 + e_2 \sim 3 e_2 + e_1$, so that, for example, $\etasym \sim 2e_1 + e_2 - e_3$.

\begin{theorem}\label{KleinQuarticThetaHeptagon}
	There is a heptagon in $\P^2$ whose adjoint is the Klein quartic $C$ and such that the associated theta characteristic is $\etasym$, the unique even theta characteristic of $C$ invariant under all automorphisms of $C$.
\end{theorem}

The construction of an associated heptagon uses the symmetries of $C$ as follows. Let $\zeta \in \C$ be a primitive seventh root of unity and consider the 7-fold symmetry $\varphi \in \operatorname{Aut}(C)$ given by 
\[ \varphi : [x,y,z] \longmapsto [\zeta^4 x, \zeta^2 y, \zeta z] \] 
For any point $p \in \P^2$, write $p_i = \varphi^{i-1}(p)$ for $i=1,\dots,7$. 

\begin{lemma}\label{lemma:heptagonKlein}
Let $p_1\in C\setminus \{e_1,e_2,e_3\}$ be a point such that the lines $\ol{p_1 p_3}$ and $\ol{p_5 p_7}$ intersect in a point on the curve $C$. Then the heptagon given by the lines
\[ L_1 = \ol{p_6p_1},\,\, L_2 = \ol{p_7p_2}, \,\, L_3 = \ol{p_1p_3}, \,\, L_4 = \ol{p_2p_4}, \,\, L_5 = \ol{p_3p_5}, \,\, L_6 = \ol{p_4p_6}, \,\, L_7 = \ol{p_5p_7} \]
has the curve $C$ as adjoint.	
\end{lemma}
\begin{proof}
	We show that the $14$ points $p_{ij}$ for $j\neq i-1,i+1$ lie on the Klein quartic. These points form two orbits under the symmetry $\varphi$ of order $7$ of the Klein quartic, namely $p_1, \ldots,p_7$ and the orbit of the intersection point of $L_3$ and $L_7$, which lies on $C$ by assumption.
\end{proof} 

\begin{proof}[{Proof of \Cref{KleinQuarticThetaHeptagon}}]
	We compute the intersection point $\ol{p_1p_3}\cap \ol{p_5p_7}$ as follows. If we write $q = \lambda_1 p_1 + \mu_1 p_3 = \lambda_2 p_5 + \mu_2 p_7$ with $\lambda_i,\mu_i\in \C$, then $\mu_2 p_7 = \lambda_1 p_1 + \mu_1 p_3 - \lambda_2 p_5$. So, with points written as column vectors, we have 
	\[ q=\begin{pmatrix}
	p & \varphi^2(p) & \mathbf{0} 
	\end{pmatrix} \cdot \begin{pmatrix}
	p & \varphi^2(p) & \varphi^4(p) 
	\end{pmatrix}^{-1} \cdot \varphi^6(p).
	\]
	We want this intersection point $q$ to lie on $C$, as well. The map $\psi\colon \P^2 \to \P^2$,
	\[ p \mapsto \begin{pmatrix}
	p & \varphi^2(p) & \mathbf{0} 
	\end{pmatrix} \cdot \begin{pmatrix}
	p & \varphi^2(p) & \varphi^4(p) 
	\end{pmatrix}^{-1} \cdot \varphi^6(p) \]
	is given in coordinates by 
	\[ p \longmapsto 
	\begin{pmatrix} (\zeta^4+2\zeta^2-2\zeta-\zeta^6)x \\ 
	(\zeta^2+2\zeta-2\zeta^4-\zeta^3)y \\
	(\zeta+2\zeta^4-2\zeta^2-\zeta^5)z \end{pmatrix}.
	\]
	In particular, this shows that the map is well-defined everywhere, also at the points where the matrix $(p \,\, \varphi^2(p) \,\, \varphi^4(p))$ is not invertible.
	The candidate points that can give rise to heptagons with adjoint $C$ are the points in $C\cap \psi^{-1}(C)$. The curve $\psi^{-1}(C)$ is a plane quartic curve given by an equation
	\[ g = a_1 x^3 y + a_2 y^3 z+ a_3 z^3 x \]
	where $a_1$, $a_2$, and $a_3$ are pairwise distinct, non-zero complex numbers (see the following \Cref{example:KleinHeptagon} for details). In terms of these coefficients, we can express the intersection $C\cap \psi^{-1}(C)$ as the divisor 
	\[ 3(e_1 + e_2 + e_3) + R_1 + R_2 + R_3 + R_4 + R_5 + R_6 + R_7 \]
	on $C$, with coordinates $R_i = [b \cdot \zeta^3 \cdot \alpha_i^3, \alpha_i,1]$ where $b = (a_1 - a_2)/(a_3-a_1)$ and $\alpha_1,\dots,\alpha_7$ are the seven roots of
	\[ t^7  = - \frac{b+1}{b^3}.\]
	Note that the intersection points $R_i$ form one orbit under the symmetry $\varphi$, as they should. By the construction of Lemma \ref{lemma:heptagonKlein}, the orbit of points $R_1, R_2, \ldots, R_7$ gives rise to a heptagon with adjoint $C$.
	
	To compute the associated theta characteristic on $C$, we go back to \Cref{lem:thetachar}.
	The theta characteristic is determined by the sum of the residual points. By construction, these are $R_1,R_2,\ldots,R_7$, and the points $Q_i = \psi(R_i)$. So we need to show that 
	\[ \sum_{i=1}^7 R_i + \sum_{j=1}^7 Q_j = 3H + \etasym. \]
	This follows from the fact that $\psi^{-1}(C).C = 3(e_1 + e_2 + e_3) + \sum_{i=1}^7 R_i$, which also implies that $\psi(C).C = 3(e_1 + e_2 + e_3) + \sum_{j=1}^7 Q_j$. Adding these two relations shows
	\[ 8 H \sim 6(e_1 + e_2 + e_3) + \sum_{i=1}^7 R_i + \sum_{j=1}^7 Q_j \sim 3 (H + \etasym) + \sum_{i=1}^7 R_i + \sum_{j=1}^7 Q_j, \]
	which simplifies to $3H + \etasym = \sum_{i=1}^7 R_i + \sum_{j=1}^7 Q_j$ as desired.
\end{proof}

\begin{example}\label{example:KleinHeptagon}
	We can produce an explicit heptagon whose adjoint is the Klein quartic following the proof above: Symbolic computation yields $g=f\left(\psi(x,y,z)\right)=a_1 x^3 y + a_2 y^3 z+ a_3 z^3 x$, where
	\begin{align*}
		a_1 &= -7 \left(1+\zeta\right)\left(2-11 \zeta+6 \zeta^2-\zeta^3-4 \zeta^4+9 \zeta^5\right)\\
		a_2 &= 7 \left(1+\zeta\right)\left(-2-3 \zeta+8 \zeta^2+\zeta^3-10 \zeta^4+5 \zeta^5\right)\\
		a_3 &= 7 \left(1+\zeta\right)\left(-2-3 \zeta-6 \zeta^2+\zeta^3+4 \zeta^4+5 \zeta^5\right).
	\end{align*}
	The points $R_1,\dots,R_7$ are given by $R_i=[b\cdot\zeta^3,\alpha_i^3,\alpha_i,1]$, where 
	$\alpha_i$ is a root of $t^7+(b+1)/b^3$ and $b=(a_1-a_2)/(a_3-a_1)$ is a root of the irreducible cubic
	\[
	   t^3+t^2-2t-1.	
	\]
	The labelled lines are then given as
	\[ L_1 = \ol{R_6R_1},\,\, L_2 = \ol{R_7R_2}, \,\, L_3 = \ol{R_1R_3}, \,\, L_4 = \ol{R_2R_4}, \,\, L_5 = \ol{R_3R_5}, \,\, L_6 = \ol{R_4R_6}, \,\, L_7 = \ol{R_5R_7} 
	\]
	by Lemma \ref{lemma:heptagonKlein}.

	We will also need to know that the adjoint map is unramified at this point. To see this, we rewrite the adjoint in the form
	\[
	   	{\rm adj}\colon \left\{\begin{array}{ccc}
			|\sO_{\P^2}(H)|^7 & \to & |\sO_{\P^2}(4H)|\\
			(L_1,\cdots,L_7) & \mapsto & \sum_{i=2}^6 \det(L_1|L_i|L_{i+1})\cdot\prod_{j\neq 1,i,i+1} L_j
		\end{array}\right.
	\]
	where, by abuse of notation, the $3\times 3$-matrices $(L_1|L_i|L_{i+1})$ on the right are formed by taking the coefficients of the lines $L_1,\dots,L_7$ as column vectors. This formula is well known can be verified directly by evaluating at the residual points of the line arrangement, see also \cite{zbMATH00989549}*{Section~3.2} and \cite{KRzbMATH07286075}*{Formula~(1.1)}. Note also that the expression on the right is multilinear and hence compatible with rescaling of $L_1,\dots,L_7$.
	We compute the Jacobian of this map, viewed as a map $(\C^3)^7\to\C^{15}$, at the heptagon above. It is a $15\times 21$-matrix, which is indeed of rank $15$. We verified through symbolic computation that the $15\times 15$-minor with columns indexed by $1, 2, 3, 4, 5, 7, 8, 10, 11, 13, 14, 16, 17, 19, 20$ does not vanish.\hfill$\Diamond$
\end{example}

Since the automorphism group of the Klein quartic $C$ is a subgroup of $\mathrm{PGL}_3(\C)$, it acts linearly on $\P^2$ and therefore induces an action on heptagons. Concretely, if $\sigma \in \operatorname{Aut}(C)$, it acts on a heptagon $(L_1,L_2,\ldots,L_7)$ by sending $L_i$ to the line $\sigma(L_i)$. This action preserves incidences on $C$ meaning that if $L_i\cap L_j\subset C$, then $\sigma(L_i)\cap \sigma(L_j)\subset C$ as well. Therefore, any automorphism of $C$ maps a heptagon with adjoint $C$ to another heptagon whose adjoint is $C$. For the set of residual points, we have $\sigma.(\mathcal{R}(L_1,L_2,\ldots,L_7)) = \mathcal{R}(\sigma.(L_1,L_2,\ldots,L_7))$. 

\begin{theorem}
	(1) The automorphism group ${\rm Aut}(C)$ of the Klein quartic acts freely on the orbit of the heptagon $(L_1,\dots,L_7)$ constructed in Example \ref{example:KleinHeptagon}. The symmetry $\varphi$ of order $7$ acts by cyclic permutation, hence there are $168/7=24$ elements modulo cyclic permutation. 

	(2) The fiber of the adjoint map over the Klein quartic contains the union of two orbits as in (1), namely the orbits of $(L_1,\dots,L_7)$ and of $(L_7,\dots,L_1)$.
\end{theorem}

\begin{proof}
	By construction of $L_1,\dots,L_7$, the symmetry $\varphi$ of order $7$ acts by cyclic permutation of the labels. To show that ${\rm Aut}(C)$ acts freely, it suffices to prove that no symmetry of order 2 or 3 stabilizes the heptagon $(L_1,\dots,L_7)$. Assume to the contrary that $\sigma \colon \P^2 \to \P^2$ is a symmetry of order $2$ or $3$ that stabilizes the heptagon.
	
	Since $\sigma$ stabilizes the heptagon, it induces an element of the dihedral group $D_7$ acting on $(L_1,L_2,\ldots,L_7)$. This immediately implies that $\sigma$ cannot have order $3$ because $D_7$ has order $14$. Any element $\sigma\in D_7$ of order $2$ stabilizes an element. We may assume that $\sigma(L_1) = L_1$, so that $\sigma$ is the permutation $(7\; 2) (6\; 3) (5\; 4)$.
	Recall that the $7$-fold symmetry $\varphi$ is the cyclic permutation by the construction in \Cref{KleinQuarticThetaHeptagon}, so that
	\begin{align*}
	(1\; 2\; 3\; 4\; 5\; 6\; 7) (7\; 2) (6\; 3) (5\; 4) = (7\; 2) (6\; 3) (5\; 4) (1\; 2\; 3\; 4\; 5\; 6\; 7)^6
	\end{align*}
	shows $\varphi \sigma = \sigma \varphi^6$. This shows that $\sigma$ normalizes the cyclic subgroup $\langle \varphi \rangle$.
	Since the 3-fold symmetry $[x,y,z] \mapsto [z,x,y]$ is also in the normalizer of $\langle\varphi\rangle$, the cardinality of this normalizer is at least $2\cdot3\cdot7$. This is a contradiction to \cite{zbMATH01574981}*{\S2.1}, which says that there are eight $7$-Sylow subgroups in the symmetry group of $C$.
	
	Now let $\sigma$ be any symmetry of $C$. Since the sum of the residual points $\mathcal{R}(L_1,L_2,\ldots,L_7)$ is $3H + \etasym$, we get that the sum of $\mathcal{R}(\sigma.(L_1,L_2,\ldots,L_7))$ is $\sigma.(3H + \etasym)$, which is $3H + \etasym$. So the theta characteristic associated with the heptagon $\sigma.(L_1,L_2,\ldots,L_7)$ is also $\etasym$.
\end{proof}

We conclude by restating our main result.
\begin{theorem}\label{theorem:main}
    A general pair $(C,\eta)$ of a plane quartic curve $C$ and an even theta characteristic $\eta$ arises from precisely $14\times 24=336$ heptagons. A general plane quartic $C$ hence arises from $14\times (24\times 36)=14\times 864$ heptagons.  
\end{theorem}

\begin{proof}
	
	By the preceding theorem, the pair of the Klein quartic $C$ and the symmetric theta characteristic $\etasym$ arises from at least $14\times 24$ heptagons, where the adjoint map is unramified. This shows that Statement $(**)$ of Lemma \ref{lemma:twoconditions} is satisfied. We already knew that Statement $(*)$ is true thanks to Corollary \ref{cor:conditionone}. Hence the conclusion follows from Lemma \ref{lemma:twoconditions}.
\end{proof}

\medskip
\textbf{Acknowledgements.} Daniel Plaumann, Rainer Sinn and Jannik Wesner were partially supported by the DFG grant \enquote{Geometry of hyperbolic polynomials} (Projektnr.~426054364). Daniel Plaumann gratefully acknowledges support from Institut Henri Poincaré (UAR 839 CNRS-Sorbonne Université), and LabEx CARMIN (ANR-10-LABX-59-01) 


	\goodbreak
	\noindent\textsc{Daniele Agostini, Universität Tübingen, Germany}\\
	\href{mailto:daniele.agostini@uni-tuebingen.de}{daniele.agostini@uni-tuebingen.de}\\

	\noindent\textsc{Daniel Plaumann, Technische Universität Dortmund, Germany}\\
	\href{mailto:daniel.plaumann@math.tu-dortmund.de}{daniel.plaumann@math.tu-dortmund.de}\\

	\noindent\textsc{Rainer Sinn, Universität Leipzig, Germany}\\
	\href{mailto:rainer.sinn@uni-leipzig.de}{rainer.sinn@uni-leipzig.de}\\

	\noindent\textsc{Jannik Lennart Wesner, Technische Universität Dortmund, Germany}\\
	\href{mailto:jannik.wesner@tu-dortmund.de}{jannik.wesner@tu-dortmund.de}
\end{document}